\documentclass[12pt]{amsart}
\usepackage{latexsym}
\usepackage{amssymb}
\usepackage[colorlinks]{hyperref}

\newtheorem{theorem}{Theorem}
\newtheorem{lemma}[theorem]{Lemma}
\newtheorem{proposition}[theorem]{Proposition}

\newtheorem{question}[theorem]{Question}

\theoremstyle{definition}

\newtheorem*{thmA}{Theorem A}
\newtheorem*{thmB}{Theorem B}

\theoremstyle{remark}

\numberwithin{equation}{theorem}

\newcommand{\GL}{{\mathrm {GL}}}

\newcommand{\SL}{{\mathrm {SL}}}
\newcommand{\PSL}{{\mathrm {PSL}}}

\newcommand{\Ker}{\operatorname{Ker}}
\newcommand{\Aut}{{\mathrm {Aut}}}

\newcommand{\Irr}{{\mathrm {Irr}}}

\newcommand{\CC}{{\mathbb C}}

\newcommand{\Center}{\mathbf{Z}}

\newcommand{\acd}{\mathrm{acd}}
\newcommand{\acde}{\mathrm{acd_{even}}}

\newcommand{\bO}{{\mathbf{O}}}
\newcommand{\bZ}{{\mathbf{Z}}}
\newcommand{\Al}{\textup{\textsf{A}}}
\newcommand{\Sy}{\textup{\textsf{S}}}

\begin{document}

\title[average character degree of finite groups]
{On the average character degree\\ of finite groups}
\thanks{The first author's research was supported by the  Spanish
Ministerio de Ciencia y Tecnolog\'{\i}a, grant MTM2010-15296 and
PROMETEO/Generalitat Valenciana.}

\author{Alexander Moret\'{o}}
\address{Departament d'\`{A}lgebra, Universitat de Val\`{e}ncia, 46100 Burjassot, Val\`{e}ncia, Spain} \email{alexander.moreto@uv.es}

\author{Hung Ngoc Nguyen}
\address{Department of Mathematics, The University of Akron, Akron,
Ohio 44325, USA} \email{hungnguyen@uakron.edu}



\date{\today}

\begin{abstract}
We prove that if the average of the degrees of the irreducible
characters of a finite group $G$ is less than $16/5$, then $G$ is
solvable. This solves a conjecture of I.\,M.~Isaacs, M.~Loukaki and
the first author. We discuss related questions.
\end{abstract}

\maketitle


\section{Introduction}

The interaction between character degrees and the structure of
finite groups has been a topic of interest for a long time, dating
back to the influential work of I.\,M.~Isaacs and
D.\,S.~Passman~\cite{Isaacs-Passman1,Isaacs-Passman2} in the
sixties. In this line of research, one often considers an invariant
concerning character degrees and studies how it reflects or is
reflected by the group. One of those invariants, namely the
\emph{character degree sum}, has been studied substantially by many
authors. We refer the reader to Chapter~11 of~\cite{Berkovic-Zhmud1}
for an overview of this topic.

Another invariant closely related to the character degree sum is the
so-called \emph{average character degree}, which has attracted
considerable interest recently. For a finite group $G$, let
$\Irr(G)$ denote the set of all irreducible (complex) characters of
$G$ and
\[\acd(G):=\frac{1}{|\Irr(G)|}\sum_{\chi\in\Irr(G)}\chi(1)\] denote the
average character degree of $G$. Answering affirmatively a question
of Y.\,G.~Berkovich and E.\,M.~Zhmud'~\cite{Berkovic-Zhmud2},
K.~Magaard and H.~Tong-Viet~\cite{mt} proved that if $\acd(G)<2$
then $G$ is solvable. This was improved by I.\,M.~Isaacs, M.~Loukaki
and the first author in~\cite{Isaacs-Loukaki-Moreto} where the
hypothesis was weakened to $\acd(G)<3$.
In~\cite{Isaacs-Loukaki-Moreto}, the authors also obtained the
connection between the average character degree and other important
characteristics of finite groups such as nilpotency and
supersolvability. These results demonstrate that the structure of a
finite group is somehow controlled by its average character degree.

From the observation that $\Al_5$ is the smallest nonsolvable group
and $\acd(\Al_5)=16/5$, it was conjectured
in~\cite{Isaacs-Loukaki-Moreto} that there is no nonsolvable group
whose the average character degree is smaller than $16/5$. The main
goal of this paper is to settle this problem.

\begin{thmA}\label{main theorem 1} Let $G$ be a finite group. If $\acd(G)<16/5$, then
$G$ is solvable.
\end{thmA}

We have seen that the set of \emph{all} irreducible character
degrees of a group encodes a lot of information about the group. In
fact, if we just consider the degrees of certain irreducible
characters (like the rational/real characters or the characters of
either $p'$-degree or degree divisible by $p$, where $p$ is a fixed
prime), then we can also obtain structural information of the group.
Along these lines, we prove the following result, where $\acde(G)$
denotes the average of the degrees of the irreducible characters of
$G$ of even degree and we adopt the convention that the average of
the empty set is $0$.

\begin{thmB}\label{main theorem 2} Let $G$ be a finite group. If $\acde(G)<18/5$, then
$G$ is solvable.
\end{thmB}

Note that $\acde(\SL_2(5))=18/5$, so this result is best possible too.

There does not seem to be any interesting analog of Theorem~B for
odd primes. Let $p$ be any prime and we write $\acd_p(G)$ to denote
the average degree of the characters of $G$ of degree divisible by
$p$. For $p>3$, we have $\acd_p(\PSL_2(p))=p$ and for $p=3$, we have
$\acd_3(\Al_5)=3$, which are as small as they can be among the
groups with some character of degree divisible by $p$.

On the other hand, it seems reasonable to expect that if $\acd_{p'}$
is the average of the degrees of the characters of $p'$-degree,
$p\neq 3$, and $\acd_{p'}(G)<\acd_{p'}(\Al_5)$ then $G$ is solvable.
This is false for $p=3$ because
$3=\acd_{3'}(\SL_2(5))<\acd_{3'}(\Al_5)$ and therefore we conjecture
that if $\acd_{3'}(G)<3$ then $G$ is solvable. We will come back to
this at the end of the paper.

Theorems~A and~B are respectively proved in Sections~\ref{section 2}
and~\ref{section3}. Our proofs, like those
in~\cite{mt,Isaacs-Loukaki-Moreto}, depend on the classification of
finite simple groups.

To end this introduction, we note that the corresponding problem for
conjugacy class sizes is more elementary and we refer the reader
to~\cite{Isaacs-Loukaki-Moreto} for a more detailed discussion.


\section{Proof of Theorem A}\label{section 2}

We begin with a lemma whose proof depends on the classification of
finite simple groups.

\begin{lemma}\label{lemma MagaardTongViet} Let $N$ be a non-abelian
minimal normal subgroup of a finite group $G$. Then there exists an
irreducible character $\theta$ of $N$ such that $\theta(1)\geq 5$
and $\theta$ is extendible to $G$. Furthermore, if $N$ is not a direct
product of copies of  $\Al_5$, then $\theta$ can be chosen so that
$\theta(1)\geq 6$.
\end{lemma}

\begin{proof} The first assertion appears
in~\cite[Theorem~1.1]{mt}. The second one essentially follows the
proof there. Suppose that $N$ is a direct product of $k$ copies of a
non-abelian simple group $S\neq \Al_5$. First we show that there
exists $\lambda\in\Irr(S)$ such that $\lambda(1)\geq 6$ and
$\lambda$ extends to $\Aut(S)$. It is routine to check this for
simple sporadic and the Tits groups by using~\cite{Atl1}. For the
alternating groups $\Al_n$ with $n\geq7$, we choose $\lambda$ to be
the character of degree $n-1$ corresponding to the partition
$(n-1,1)$. Finally, for simple groups of Lie type (different from
$\Al_5$) in characteristic $p$, the required character $\lambda$ can
be chosen to be the Steinberg character of $S$ of degree $|S|_p\geq
6$. (see~\cite{Feit} for the fact that the Steinberg character of
$S$ is extendible to $\Aut(S)$.) Now we take $\theta\in\Irr(N)$ to
be the tensor product of $k$ copies of $\lambda$ and
apply~\cite[Lemma~5]{Bianchi-Lewis} to conclude that $\theta$
extends to $G$.
\end{proof}

Note that we could weaken the hypothesis of the second part of the
lemma to $N\neq \Al_5$.

Now, we set some notation that we will use throughout the paper.
Given a finite group $G$, we write $\Irr(G)$ to denote the set of
irreducible characters of $G$. If $d$ is an integer, then $n_d(G)$
stands for the number of irreducible characters of $G$ of degree $d$
and if $N\trianglelefteq G$, then \[\Irr(G|N):=\{\chi\in\Irr(G)\mid
N\not\subseteq\Ker(\chi)\}\] and
\[n_d(G|N):=|\{\chi\in\Irr(G|N))\mid \chi(1)=d\}|.\] We also write
$\acd(G|N)$ to denote the average degree of the characters in
$\Irr(G|N)$. Also, if $\theta\in\Irr(N)$ we write $\Irr(G|\theta)$
to denote the set of irreducible characters of $G$ that lie over
$\theta$ and $\acd(G|\theta)$ to denote the average degree of the
members of $\Irr(G|\theta)$.

The following elementary lemma on characters of central products will be used several times.

\begin{lemma}\label{cp}
Let $G=MC$ be a central product, and let $Z=M\cap C$, so $Z\leq\Center(G)$.
If $\lambda\in\Irr(Z)$ then $\acd(G|\lambda)=\acd(M|\lambda)\cdot\acd(C|\lambda)$.
\end{lemma}

\begin{proof}
This follows from the fact that there exists a bijective
correspondence between
$\Irr(M|\lambda)\times\Irr(C|\lambda)\longrightarrow\Irr(G|\lambda)$
such that if $(\alpha,\beta)$ is mapped to $\chi$ then
$\chi(1)=\alpha(1)\beta(1)$.
\end{proof}

Let $G$ be a finite group with $\acd(G)<16/5$. If $G$ is abelian
then we are done, so it suffices to assume that $G$ is non-abelian.
Now let $N\subseteq G'$ be a minimal normal subgroup of $G$. The
proof of Theorem~A is divided in two fundamentally different cases:
$N$ is non-abelian and $N$ is abelian. The following result handles
the former case.

\begin{proposition}\label{proposition 1} Let $G$ be a finite group. If $G$ has a non-abelian
minimal normal subgroup $N$ such that $N\subseteq G'$, then
$\acd(G)\geq 16/5$.
\end{proposition}

\begin{proof} Assume, to the contrary, that the assertion is false and let $G$ be a minimal counterexample. In particular, we have $\acd(G)<3.2$.
Recall that $n_d(G)$ denotes the number of irreducible characters of
$G$ of degree $d$. We then have \[\acd(G)=\frac{\sum_d
dn_d(G)}{\sum_d n_d(G)}<16/5,\] or equivalently,
\begin{equation}\label{equation1}\sum_{d\geq 4} (5d-16)n_d(G)<11n_1(G)+6n_2(G)+n_3(G).\end{equation}

Since $N$ is a non-abelian minimal normal subgroup of $G$, we know
that $N$ is isomorphic to a direct product of copies of a
non-abelian simple group, say
\[N=\underbrace{S\times S\times\cdots\times S}_{k \text{ times}},\]
where $S$ is non-abelian simple. From the hypothesis $N\subseteq
G'$, we have that $N$ is contained in the kernel of every linear
character of $G$. We now show that $N$ is contained in the kernel of
every irreducible character of degree $2$ of $G$ as well. So let
$\chi\in\Irr(G)$ with $\chi(1)=2$. Since $N$ has no irreducible
character of degree 2 (see Problem~3.3 of~\cite{isa}) and has only
one linear character, which is the trivial one, it follows that
$\chi_N=2\cdot1_N$. We then have $N\subseteq \Ker(\chi)$, as wanted.
We have shown that $n_d(G)=n_d(G/N)$ for $d=1$ and $2$. Now, we
split the proof in three cases.

\medskip

1) First, we assume that $S$ is different from $\Al_5$ and
$\PSL_2(7)$. From the classification of the finite irreducible
subgroups of $\GL(3,\CC)$, we know that $\Al_5$ and $\PSL_2(7)$ are
the only simple groups possessing an irreducible character of degree
$3$. Therefore, every non-principal character of $N$ has degree at
least 4 and, by using the same arguments as above, we deduce that
$N$ is contained in the kernel of every irreducible character of
degree 3 of $G$. Together with the conclusion of the previous
paragraph, we obtain
\[n_1(G)+n_2(G)+n_3(G)=n_1(G/N)+n_2(G/N)+n_3(G/N)\leq |\Irr(G/N)|.\]
By Lemma~\ref{lemma MagaardTongViet}, there exists an irreducible
character $\theta$ of $N$ such that $\theta(1)\geq 6$ and $\theta$
extends to a character $\psi\in\Irr(G)$. It follows by Gallagher's
theorem (Corollary~6.17 of~\cite{isa}) that the map $\beta\mapsto
\beta\psi$ is an injection from $\Irr(G/N)$ to the set of
irreducible characters of $G$ with degree at least 6. So we have
$$|\Irr(G/N)|\leq \sum_{d\geq6} n_d(G).$$

The last two inequalities yield
\[n_1(G)+n_2(G)+n_3(G)\leq \sum_{d\geq6} n_d(G).\] It follows that
\[11n_1(G)+6n_2(G)+n_3(G)\leq 11(n_1(G)+n_2(G)+n_3(G))\leq \sum_{d\geq6}
11n_d(G),\] which violates the inequality~(\ref{equation1}).

\medskip

2) Next we consider $N=\Al_5\times \Al_5\times\cdots\times \Al_5$, a
direct product of $k$ copies of $\Al_5$. We note that $\Al_5$ has an
irreducible character of degree 5 that is extendible to
$\Aut(\Al_5)=\Sy_5$. Using~\cite[Lemma~5]{Bianchi-Lewis}, we obtain
that $N$ has an irreducible character of degree $5^k$ that extends
to a character $\psi\in\Irr(G)$. Using Gallagher's theorem again and
noting that $n_d(G)=n_d(G/N)$ for $d=1$ and $2$, we have
\[n_1(G)+n_2(G)+n_3(G/N)\leq n_{5^k}(G)+n_{2.5^k}(G)+n_{3.5^k}(G).\]
This implies that
\begin{align*}9n_1(G)+6n_2(G)+n_3(G/N)&\leq
9(n_1(G)+n_2(G)+n_3(G/N))\\
&\leq 9(n_{5^k}(G)+n_{2.5^k}(G)+n_{3.5^k}(G))\end{align*} Therefore
\begin{equation}\label{equation2} 9n_1(G)+6n_2(G)+n_3(G/N)\leq \sum_{d\geq 5^k} (5d-16)n_d(G).\end{equation}

Now we need to bound the number of irreducible characters of $G$ of
degree 3 whose kernels do not contain $N$. We claim that
\begin{equation}\label{equation3}n_3(G|N)\leq 2kn_1(G).\end{equation} To see this,
let $\chi\in\Irr(G)$ such that $\chi(1)=3$ and $N\nsubseteq
\Ker(\chi)$. Since $N$ has no non-principal linear character, the
restriction $\chi_N$ must be irreducible. By Gallagher's theorem,
the number of irreducible characters of $G$ of degree 3 lying over
$\chi_N$ equals to $n_1(G/N)=n_1(G)$. Note that $\chi_N$ has degree
3 and $N$ has exactly $2k$ irreducible characters of degree 3. We
conclude that the number of irreducible characters of $G$ of degree
3 whose kernels do not contain $N$ is at most $2kn_1(G)$, as
claimed.

Combining~\eqref{equation2} and~\eqref{equation3}, we obtain
\begin{equation}\label{equation4} 9n_1(G)+6n_2(G)+n_3(G)\leq 2kn_1(G)+\sum_{d\geq 5^k} (5d-16)n_d(G).\end{equation}

We observe that $\Al_5$ has an irreducible character of degree 4
that is extendible to $\Sy_5$. Arguing similarly as above, we have
\[n_{4^k}(G)\geq n_1(G).\] This and~\eqref{equation4} yield
\begin{align*} \sum_{d\geq 4} (5d-16)n_d(G)&\geq
(5.4^k-16)n_{4^k}(G)+\sum_{d\geq 5^k} (5d-16)n_d(G)\\
&\geq (5.4^k-16)n_{1}(G)+ 9n_1(G)+6n_2(G)+n_3(G)-2kn_1(G)\\
&\geq 11n_1(G)+6n_2(G)+n_3(G),
\end{align*}
where the last inequality comes from the fact that $5\cdot4^k-16\geq
2+2k$ for every $k\geq 1$. Because of~\eqref{equation1}, we obtain a
contradiction, which completes the proof of this case.

\medskip

3) Finally we assume that $N=\PSL_2(7)\times
\PSL_2(7)\times\cdots\times \PSL_2(7)$, a direct product of $k$
copies of $\PSL_2(7)$. We observe that $\PSL_2(7)$ has exactly two
irreducible characters of degree $3$. Moreover, it also has
irreducible characters of degrees 7 and 8 that are extendible to
$\Aut(\PSL_2(7))$. The proof now follows the same lines as in~2).
\end{proof}

We are now ready to complete the proof of Theorem~A. Given a group
$G$, we write $\bO_\infty(G)$ to denote the largest normal solvable
subgroup of $G$.

\begin{proof}[Proof of Theorem A] Assume, to the
contrary, that the theorem is false, and let $G$ be a minimal
counterexample. In particular, the group $G$ is non-solvable and
$\acd(G)<16/5$. As in the previous proposition,  this inequality is equivalent to
\begin{equation}\label{equation21}\sum_{d\geq 4}
(5d-16)n_d(G)<11n_1(G)+6n_2(G)+n_3(G).\end{equation} Let $M\lhd G$
be minimal such that $M$ is non-solvable. We note that $M$ is
perfect and contained in the last term of the derived series of $G$.
Let $N\subseteq M$ be a minimal normal subgroup of $G$. In addition,
we take $N\leq [M,\bO_{\infty}(M)]$ if $[M,\bO_{\infty}(M)]\neq1$
and, if possible, we take $N$ to be of order $2$. We then have $N
\subseteq M=M'\subseteq G'$.

In light of Proposition~\ref{proposition 1}, we may assume that $N$
is abelian. It follows that the quotient $G/N$ is non-solvable, and
by the minimality of $G$, we deduce that $\acd(G/N)\geq 16/5$. In
other words, we have
\[\sum_{d\geq 4}
(5d-16)n_d(G/N)\geq11n_1(G/N)+6n_2(G/N)+n_3(G/N).\] Since $N$ is
contained in the kernel of every linear character of $G$, we have
$n_1(G/N)=n_1(G)$ and hence it follows from~\eqref{equation21} that
\[6n_2(G/N)+n_3(G/N)<6n_2(G)+n_3(G).\] As before, if $\Irr(G|N)$
denotes the set of irreducible characters of $G$ whose kernels do
not contain $N$, we see that there is $\chi\in\Irr(G|N)$ such that
$\chi(1)=2$ or $3$.

Let $K:=\Ker(\chi)$. Since $N\nsubseteq K$, we have $M\nsubseteq K$,
and hence $MK/K$ is nontrivial. Also, recall that $M$ is perfect. We
deduce that $MK/K$ is non-solvable and hence so is $G/K$. It then
follows that $\chi$ is primitive. (In order to see this, we write
$\overline{G}=G/K$ and use the bar convention. We also view $\chi$
as a (faithful) character of $\overline{G}$. If $\chi$ is
imprimitive, it is induced from a linear character of a subgroup
$\overline{H}$ with index $\chi(1)\leq3$, and thus
$\overline{G}/\overline{N}$ is solvable, where $\overline{N}={\rm
core}_{\overline{G}}(\overline{H})$. Now
$\overline{N}\trianglelefteq\overline{G}$ and $\chi_{\overline{N}}$
has a linear constituent, and thus all irreducible constituents of
$\chi_{\overline{N}}$ are linear. Since $\chi$ is faithful,
$\overline{N}$ is abelian, and thus $\overline{G}$ is solvable.)
Therefore $G/K$ is a primitive linear group of degree $2$ or $3$.
Let $C/K=\Center(G/K)$. By the classification of the primitive
linear groups of degrees $2$ and $3$ (see Section~81 of
\cite[Chapter~V]{bli}), we have
\[G/C\cong \Al_5, \Al_6, \text{ or } \PSL_2(7).\] Moreover, the
cases $G/C\cong \Al_6$ or $\PSL_2(7)$ only possibly happen when
$\chi(1)=3$.

We will utilize some arguments in the proof
of~\cite[Theorem~2.2]{Isaacs-Loukaki-Moreto} to show that $G=MC$ is
a central product with a central amalgamated subgroup
$\Center(M)=M\cap C$.

First, since $C/K=\Center(G/K)$ is abelian and $MK/K$ is
non-solvable, we have $M\nsubseteq C$, and as $G/C$ is simple, we
deduce that $G=MC$. It follows that \[\frac{M}{M\cap C}\cong
\frac{MC}{C}=\frac{G}{C}.\] Moreover, as $M\cap C\lhd G$ is a proper
subgroup of $M$, we have $M\cap C\leq \bO_\infty(M)$ by the
minimality of $M$.  Since $M/(M\cap C)$ is simple, we now deduce
that $M\cap C=\bO_\infty(M)$ and hence $\bO_\infty(M)\subseteq C$.
Thus $[M,\bO_\infty(M)]\subseteq K$. But $N\nsubseteq K$, so we have
$N\nsubseteq [M,\bO_\infty(M)]$. By the choice of $N$, we have
$\bO_\infty(M)=\Center(M)$. We have proved that $M$ is a perfect
central cover of the simple group $M/\Center(M)\cong M/(M\cap
C)\cong G/C$. Moreover, as $N\subseteq M$ is abelian, we have
$1<N\subseteq \Center(M)$.

Since $C$ and $M$ are both normal in $G$, we have
$[M,C]=[C,M]\subseteq C\cap M=\Center(M)$ so that
$[C,M,M]=[M,C,M]=1$. By the three subgroups lemma, we have $[M,M,C]=1$
and hence $[M,C]=1$ as $M$ is perfect. We conclude that $G=MC$ is a
central product with a central amalgamated subgroup $M\cap
C=\Center(M)>1$. Now, we consider separately each of the three possibilities for $G/C$.

\medskip

\textbf{Case} $G/C\cong \Al_5$. Then, as $\Al_5$ has only one
(non-simple) perfect central cover, namely $\SL_2(5)$, we have
$M\cong \SL_2(5)$. Using Lemma~\ref{cp}, we see that
$\acd(G|\lambda)\geq\acd(M|\lambda)$ for every linear character
$\lambda$ of $\bZ(M)\cong C_2$. If $\lambda$ is trivial, this yields
\[\acd(G|\lambda)\geq \acd(M/\bZ(M))=\acd(\Al_5)=16/5.\]
On the other hand, if $\lambda$ is the only nonprincipal linear
character of $\bZ(M)$, we have
\[\acd(G|\lambda)\geq \acd(M|\lambda)=(2+2+4+6)/4=14/4.\]
We deduce that $\acd(G)\geq 16/5$ and this violates our assumption.

%

\medskip

\textbf{Case} $G/C\cong \PSL_2(7)$. We will see that this case
cannot happen. Since $G/C\cong M/\bZ(M)$ and $1<\bZ(M)$ we deduce
that $M\cong \SL_2(7)$. It follows that $N=\bZ(M)$ has order $2$.
Also, $\chi_M$ is irreducible of degree $3$, and a check of the
character table of $\SL_2(7)$ shows that $N$ is in the kernel of
$\chi$. This is a contradiction.

\medskip

\textbf{Case} $G/C\cong \Al_6$. A check of the character tables of
the various covers of $\Al_6$ shows that if $N$ has order $2$ then
it is in the kernel of the irreducible characters of $M$ of degree
$3$. Since $N$ was chosen to have order $2$ if possible, it follows
that $2$ does not divide $|\bZ(M)|$, so the only possibility in this
case is $|\bZ(M)|=3$.

Now, by Lemma~\ref{cp} we see that
$\acd(G|\lambda)\geq\acd(M|\lambda)$ for all linear characters
$\lambda$ of $\bZ(M)$. If $\lambda$ is principal, this yields
\[\acd(G|\lambda)\geq\acd(M/\bZ(M))=\acd(\Al_6)=46/7.\] If $\lambda$ is not
principal then
$$
\acd(G|\lambda)\geq\acd(M|\lambda)=(3+3+6+9+15)/5=36/5.
$$
Since $\acd(G|\lambda)\geq16/5$ in all cases, it follows that
$\acd(G)\geq16/5$. This final contradiction completes the proof.
\end{proof}

Our Theorem~A together with Theorems~B and~C
of~\cite{Isaacs-Loukaki-Moreto} provide an optimal sufficient
condition on the average character degree for a finite group to be
respectively solvable, supersolvable, and nilpotent. One might want
to know whether there is a sufficient condition on the average
character degree for a finite group to be $p$-solvable where $p$ is
a prime. It was proved in~\cite[Theorem~1]{mn} that a finite group
$G$ is $p$-solvable whenever
$\sum_{\chi\in\Irr(G)}\chi(1)>(\sqrt{3}/p)|G|$. In view of a
consequence of the Cauchy-Schwarz inequality that $\acd(G)\leq
|G|/\sum_{\chi\in\Irr(G)}\chi(1)$, we put forward the following

\begin{question} Is it true that a finite group $G$ must be
$p$-solvable if $\acd(G)<p/\sqrt{3}$?
\end{question}


\section{Characters of even degree and other
variations}\label{section3}

We start proving Theorem B, which we restate.

\begin{theorem}
 Let $G$ be a finite group. If $\acde(G)<18/5$, then
$G$ is solvable.
\end{theorem}

\begin{proof} By way of contradiction, assume that
$G$ is non-solvable with $\acde(G)<3.6$. Let $M\lhd G$ be minimal
such that $M$ is non-solvable. Then $M$ is perfect and contained in
the last term of the derived series of $G$. Let $N\subseteq M$ be a
minimal normal subgroup of $G$. We then have $N \subseteq
M=M'\subseteq G'$.

Suppose first that $N$ is non-abelian. By Lemma~\ref{lemma
MagaardTongViet}, there exists $\theta\in \Irr(N)$ such that
$\theta(1)\geq 5$ and $\theta$ extends to $\psi\in\Irr(G)$. We
deduce from Gallagher's theorem that there is an injection from
$\Irr(G/N)$ to the set of irreducible characters of $G$ of degree
at least 5.  We have
\[n_2(G/N)\leq \sum_{d\geq 10 \text{ even}}n_d(G).\]
We have already seen in the proof of Proposition~\ref{proposition 1}
that $N$ is contained in the kernel of every irreducible character
of $G$ of degree 2. In other words, we have $n_2(G)=n_2(G/N)$ and it
follows that
\[n_2(G)\leq \sum_{d\geq 10 \text{ even}}n_d(G).\] Hence
\[8n_2(G)\leq \sum_{d\geq 10 \text{ even}}8n_d(G)\leq \sum_{d\geq 4 \text{ even}}(5d-18)n_d(G),\]
which in turn yields
\[\sum_{d \text{ even}}18n_d(G)\leq \sum_{d \text{ even}}5dn_d(G).\]
It follows that
\[\acde(G)=\frac{\sum_{d \text{ even}}dn_d(G)}{\sum_{d \text{ even}}n_d(G)}\geq
\frac{18}{5}=3.6,\] and this violates the hypothesis.

Now we can assume that $N$ is abelian. Since $G$ is non-solvable,
the quotient $G/N$ is non-solvable as well. By the induction
hypothesis, we have $\acde(G/N)\geq 3.6$. As $\acde(G)<3.6$, we
deduce that $n_2(G)>n_2(G/N)$, and hence there exists
$\chi\in\Irr(G)$ such that $\chi(1)=2$ and $N\nsubseteq \Ker(\chi)$.
Following the proof of Theorem A, we can show
that $G=MC$ is a central product with the amalgamated subgroup
$Z:=\Center(M)=N\cong C_2$, where
$C/\Ker(\chi):=\Center(G/\Ker(\chi))$ and $M=\SL_2(5)$. Furthermore,
there is a bijection $\Irr(M|Z)\times \Irr(C|Z)\rightarrow
\Irr(G|Z)$, where if $(\alpha,\beta)\mapsto \psi$ then
$\psi(1)=\alpha(1)\beta(1)$. Employing the arguments in the proof
of~\cite[Theorem~2.2]{Isaacs-Loukaki-Moreto}, we can evaluate the
number of irreducible characters of $G$ at each even degree. Indeed,
we only need to focus on $n_2(G)$, $n_4(G)$, and $n_6(G)$.

Recall that $\chi(1)=2$ and $\chi\in\Irr(G|Z)$. Therefore, if
$(\alpha,\beta)\mapsto \chi$, we must have $\beta(1)=1$ since
$\alpha(1)\geq 2$. So $\beta\in\Irr(C|Z)$ is an extension of the
unique non-principal linear character of $Z$. By Gallagher's
theorem, we obtain a degree-preserving bijection from $\Irr(C/Z)$ to
$\Irr(C|Z)$. In particular,
\[n_1(C|Z)=n_1(C/Z)=n_1(G/M)=n_1(G).\] Moreover, as $n_1(M|Z)=0$ and
$n_2(M|Z)=2$, we have
\[n_2(G|Z)=n_1(M|Z)n_2(C|Z)+n_2(M|Z)n_1(C|Z)=2n_1(C|Z)=2n_1(G).\]
Next, as $n_1(M/Z)=1$ and $n_2(M/Z)=0$, we have
\[n_2(G/Z)=n_1(M/Z)n_2(C/Z)+n_2(M/Z)n_1(C/Z)=n_2(C/Z).\] It follows
that
\[n_2(G)=n_2(G/Z)+n_2(G|Z)=n_2(C/Z)+2n_1(G),\] which implies that
$n_2(C/Z)=n_2(G)-2n_1(G)$ and in particular $n_2(G)\geq n_1(G)$.

We now evaluate $n_4(G)$. We have
\[n_4(G|Z)\geq n_4(M|Z)n_1(C|Z)=n_1(C|Z)=n_1(G)\] since
$n_4(M|Z)=1$, and
\[n_4(G/Z)\geq n_4(M/Z)n_1(C/Z)=n_1(C/Z)=n_1(G).\]
We deduce that
\[n_4(G)=n_4(G|Z)+n_4(G/Z)\geq 2n_1(G).\]

Next, we evaluate $n_6(G)$. We have
\[n_6(G|Z)\geq n_6(M|Z)n_1(C|Z)=n_1(C|Z)=n_1(G)\] since
$n_6(M|Z)=1$, and
\[n_6(G/Z)\geq n_3(M/Z)n_2(C/Z)=2n_2(C/Z)=2(n_2(G)-2n_1(G)).\]
We also deduce that
\[n_6(G)=n_6(G|Z)+n_6(G/Z)\geq n_1(G)+2(n_2(G)-2n_1(G))=2n_2(G)-3n_1(G).\]

Combining the inequalities for $n_4(G)$ and $n_6(G)$, we have
\begin{align*}\sum_{d\geq 4 \text{ even}}(5d-18)n_d(G)&\geq 2n_4(G)+12n_6(G)\\
&\geq 2(2n_1(G))+12(2n_2(G)-3n_1(G))\\
&=24n_2(G)-32n_1(G)\\
&=8n_2(G)+16(n_2(G)-2n_1(G))\\
&\geq 8n_2(G),
\end{align*}
where the last inequality is due to the fact that $n_2(G)\geq
2n_1(G)$. Now, as above, we deduce that $\sum_{d \text{
even}}18n_d(G)\leq \sum_{d \text{ even}}5dn_d(G)$ and hence
$\acde(G)\geq 3.6$. This final contradiction completes the proof.
\end{proof}

To prove the predictions involving $\acd_{p'}$ that we mentioned in
the introduction, it would be useful to have a version of
Lemma~\ref{lemma MagaardTongViet} where we can guarantee the
existence of an extendible character of $p'$-degree. According to
Theorem~4.1 of~\cite{imn}, if $S$ is a simple group and $p>3$ then
there is an irreducible character of $S$ of $p'$-degree that is
$\Aut(S)$-invariant. Therefore it seems reasonable to hope that such
a strong form of Lemma~\ref{lemma MagaardTongViet} may be true. If
this were true then it would be enough to mimic the proof of
Theorem~A to obtain that if $p>5$ and $\acd_{p'}(G)<16/5$, then $G$
is solvable. However, this is not true for $p=2$ or $3$: it turns
out that all the non-principal characters of $\PSL_3(4)$ extendible
to $\Aut(\PSL_3(4))$ have even degree and the only non-principal
character of $\PSL_2(9)$ extendible to $\Aut(\PSL_2(9))$ is the
Steinberg character.

\section*{Acknowledgement} We are grateful to the referee
for helpful suggestions that have significantly improved the
exposition of the paper.



\begin{thebibliography}{99}

\bibitem{Bianchi-Lewis}
M. Bianchi, D. Chillag, M.\,L. Lewis, and E.~Pacifici, Character
degree graphs that are complete graphs, {\it Proc. Amer. Math.
Soc.}~{\bf135} (2007), 671-676.

\bibitem{Berkovic-Zhmud1}
Y.\,G. Berkovich and E.\,M.~Zhmud', \emph{Characters of Finite
Groups. Part 1}, Translated from the Russian manuscript by P.
Shumyatsky and V. Zobina, Translations of Mathematical Monographs
\textbf{172}, American Mathematical Society, Providence, RI, 1998.

\bibitem{Berkovic-Zhmud2}
Y.\,G. Berkovich and E.\,M.~Zhmud', \emph{Characters of Finite
Groups. Part 2}, Translated from the Russian manuscript by P.
Shumyatsky, V.~Zobina, and Y.\,G.~Berkovich. Translations of
Mathematical Monographs~\textbf{181}, American Mathematical Society,
Providence, RI, 1999.


\bibitem{bli} H.\,F. Blichfeldt, {\it Finite Collineation Groups}, The University of Chicago Press, 1917.


\bibitem{Atl1}
  J.\,H. Conway, R.\,T. Curtis, S.\,P. Norton, R.\,A. Parker, and R.\,A. Wilson,
 {\it Atlas of Finite Groups}, Clarendon Press, Oxford, 1985.

\bibitem{Feit}
W. Feit, Extending Steinberg characters, Linear algebraic groups and
their representations, {\it Contemp. Math.} {\bf 153} (1993), 1-9.

\bibitem{isa} I.\,M. Isaacs, {\it Character Theory of Finite Groups}, Dover, 1994.

\bibitem{Isaacs-Loukaki-Moreto}
I.\,M. Isaacs, M. Loukaki, and A. Moret\'{o}, The average degree of
an irreducible character of a finite group, \emph{Israel J.
Math.}~{\bf 197} (2013), 55--67.

\bibitem{imn} I.\,M. Isaacs, G. Malle, and G. Navarro, Real characters of $p'$-degree, {\it J. Algebra} {\bf 278} (2004), 611-620.

\bibitem{Isaacs-Passman1}
I.\,M. Isaacs and D.\,S.~Passman, A characterization of groups in
terms of the degrees of their characters, \emph{Pacific J.
Math.}~\textbf{15} (1965), 877--903.

\bibitem{Isaacs-Passman2}
I.\,M. Isaacs and D.\,S.~Passman, A characterization of groups in
terms of the degrees of their characters. II, \emph{Pacific J.
Math.}~\textbf{24} (1968), 467--510.

\bibitem{mt}
K. Magaard and H.\,P. Tong-Viet, Character degree sums in finite
nonsolvable groups, {\it J. Group Theory}~{\bf 14} (2011), 54-57.

\bibitem{mn} A. Mar\'{o}ti and H.\,N. Nguyen, Character degree sums of finite groups, \emph{Forum
Math.}, to appear.

\end{thebibliography}
\end{document}